\documentclass[letterpaper, 10 pt, conference]{ieeeconf}  

\IEEEoverridecommandlockouts                              
\usepackage{amssymb}
\usepackage{amsmath}

\usepackage{amsthm}
\usepackage{algorithm,algorithmic}
\usepackage{graphicx}
\usepackage{bbm}
\usepackage{cite}

\newtheorem{prop}{Proposition}
\newtheorem{lemma}{Lemma}

\usepackage{comment}

\usepackage{multirow}

\title{\LARGE \bf
Finite-Horizon  Markov  Decision Processes with \\ Sequentially-Observed Transitions
}

\author{{Mahmoud El Chamie $\ $}
and
{$\, $ Beh\c{c}et A\c{c}{\i}kme\c{s}e 
\thanks{$^{*}$ The authors are with the University of Texas at Austin, department of Aerospace Engineering and Engineering Mechanics, 210 E. 24th St., Austin, TX 78712 USA. Emails: {\em melchami@utexas.edu} and {\em behcet@austin.utexas.edu}}}
}

\begin{document}

\maketitle
\thispagestyle{empty}
\pagestyle{empty}

\begin{abstract}
Markov Decision Processes (MDPs) have been used to formulate many decision-making problems   in science and engineering. The objective is to synthesize  the best decision (action selection) policies to maximize expected rewards (or minimize costs) in a given  stochastic dynamical environment. 
In this paper, we extend this model by incorporating  additional information that  the transitions due to actions can be sequentially observed. The proposed  model benefits from this  information  and produces  policies with better  performance than those of  standard  MDPs. The paper also  presents an efficient offline  linear programming based algorithm to synthesize  optimal policies for the extended model.
\end{abstract}

\section{Introduction}

Markov Decision Processes (MDPs) have been used to formulate many decision-making problems  in a variety of  areas of science and engineering \cite{Parkes:2003Anm,Dolgov:2006Res,Doshi:2004Dyn}. MDPs have proved  useful in modeling decision-making problems for stochastic  dynamical systems where the dynamics cannot be fully captured by using first principle formulations. MDP models can be constructed by utilizing the available measured  data, which allows construction of state transition probabilities. Hence MDPs play a critical role in big-data analytics. Indeed  very popular   methods of machine learning  such as reinforcement and its variants  \cite{sutton1998introduction}\cite{szepesvari2010} are built upon the MDP framework.    With the increased interest and efforts in Cyber-Physical Systems (CPS), there is  even more interest in MDPs to facilitate rigorous construction of innovative hierarchical decision-making architectures, where MDP framework can integrate physics-based models with data-driven models.  Such decision architectures can utilize a systematic approach to bring physical devices together with software  to benefit many emerging engineering  applications, such as autonomous systems.    

In many applications \cite{feinberg2002handbook}\cite{altman:inria-00072663}, MDP models  are used to compute optimal  decisions when future actions contribute to the overall mission performance. 
Here  we consider MDP-based  stochastic decision-making models  \cite{Puterman:1994Mar}. 
An MDP model is composed of a set of  time instances (epochs), actions, states, and immediate rewards/costs. Actions transfer the system in a stochastic manner from one state to another and rewards are collected based on the actions taken at the corresponding states. Hence MDP models provide  analytical descriptions  of stochastic processes with  state and action spaces, the state transition probabilities as a function of actions, and with rewards as a function of the states and actions. The objective is to design the best decision (action selection) policies to maximize expected rewards (minimize costs) for a given MDP. 

    With the advent of Internet of Things (IoT) and the increasing sensing capabilities, increasingly large amounts of data are collected.  This paper aims to extend the typical MDP framework to exploit additional sensed information. In particular, we consider a scenario where  not only the current state of the agent is known but also  the transition due to an action can be observed in a sequential manner: The outcome of action 1 is observed and a decision is made on whether to rake the action or not, and this process is continued  until one of the actions (in the given order) is taken.   Decisions are taken at instances called \emph{phases}. A phase starts with an observation for the transition caused by an action  and ends with  a decision about whether to take this  action   or not.

MDPs have been widely studied  since the pioneering work of Bellman \cite{Bellman:1957Dyn}, which provided the foundation of dynamic programming, and the book of Howard  \cite{Howard:1960Dyn} that popularized the study of decision processes. The standard  MDP models are applied to diverse fields including robotics, automatic control, economics, manufacturing, and communication networks. There have been several extensions and generalizations of the MDP models to fit specific  application requirements and considerations into the models. Typical MDP problems assume that at every decision epoch, agents know their current state, and the reward for choosing an action, while the environment is stochastic, i.e.,  the transitions cannot be predicted in a deterministic manner.  For example  partially observed MDPs (POMDPs) extend the typical MDP problems to take into account  uncertainties in the agent  state knowledge \cite{Kaelbling:1998Pla}. There can also be uncertainties in state transition/reward models.  Learning methods are developed to handle such uncertainties  (e.g., reinforcement learning \cite{Kaelbling:1996Rei}).  In typical MDPs, decisions are taken on discrete epochs. Continuous-time MDPs \cite{Guo:2009Con} extend this model by relaxing the assumption of discrete events and models  to continuous time and  space models.  Another extension is the Bandit problem \cite{Auer:2002Fin}, where the agents  can observe the random reward of different actions and have to choose the actions that  maximize the sum of rewards through a sequence of repeated experiments. In other decisions-making problems, determination of  optimal stopping time is studied to  determine  optimal epoch for a particular action \cite[Chapter~13]{DeGroot:2004Opt}.   In other applications, multi-objective cost functions or constraints are considered for the computation of the optimal MDP policies \cite{altman1999constrained}.  

In most of the relevant literature, the extensions to  the standard  MDP models are obtained  by relaxing some of its assumptions (like observability of current state, known rewards, transition probabilities, etc.). In this paper, however, we extend typical MDP problems by considering a more general model when more information about the environment and the process  is available. This latter assumption is motivated by the fact that the evolving field of IoT is providing agents with a lot of additional data that can be utilized in the model to synthesize  better decision-making strategies. In particular, we assume that not only the current state, but the environmental transition due to possible actions are also observed in a sequential manner. We aim to build decision-making models that benefit from this class of  information  to generate policies having better total expected rewards. 

\section{Sequentially Observed MDP}
\subsection{Examples}
This section presents several motivating examples for sequentially observed MDPs.
\subsubsection{Routing}
 Consider a vehicle that aims to go to a final desired position  (or a packet if a computer network is considered) and there are  three possible routes from the current one-way street that  the vehicle is on (Fig. \ref{routing}). The current street and the exits form the shape of letter ``E", that is, if the  vehicle passes an turn then the corresponding route is ruled out.
 
 \begin{figure}
\begin{center}
\includegraphics[scale=0.25]{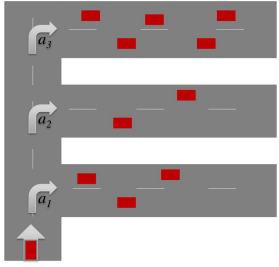}
\caption{The figure shows an example where sequential MDP models can be applied. The vehicle knows the historical data for the congestion for there separate  routes to a common destination. But once the vehicle is on the turn, it can observe the actual real-time congestion status of one route at a time with a fixed sequence of observations and only knows the expected congested status  of the upcoming routes. If the action to take the turn is rejected (route  not taken) the vehicle cannot come back.
} 
\label{routing}
\vspace{-20pt}
\end{center}
\end{figure}
 
 Each route can have congestion, for  which there is prior knowledge based on historical data.  The vehicle can only observe the current  traffic conditions  when it is at the turn. If the vehicle decided not to take one route based on the observed congestion, it cannot get back later after it observes the other route. The vehicle is forced to take one of the routes, so if it rejects all observed congested routes, it will be stuck with the last choice and should take it regardless of the route congestion status. The question here is whether the vehicle would take a route given the observed congestion and historical data for the next turn. Standard MDP models will select beforehand routes having the lowest average (expected) congestion regardless of the observed routes status. 
\subsubsection{University Admission}
Suppose that a university has a certain number of scholarships for a program. Applicants are interviewed in a sequential manner on different selection rounds, i.e., in the first round the candidates are interviewed, evaluated, and admission decisions are announced before other candidates can apply in the second round. If a student is granted a scholarship, the available funding is decreased and the system changes its state. The rewards obtained are assumed to be the evaluation of the profile of the selected candidates assuming all applicants can be evaluated and compared by a scoring function. The committee knows what the average score of applicants would be at different rounds (based prior data). Note that the evaluation committee can observe the profile of candidates at a given round, but they only know the average profile score of the next rounds. The question in this scenario is: Given the current (observed) applicant pool and expected pool for the next rounds, how many of the applicants in this round should be accepted? 
If we use a standard MDP model, then the solution would be to
select all candidates from the round with the highest average. Clearly this solution is not practical in this scenario because very important information are being discarded  and a better approach must be used.

\subsubsection{Market Investment}
Another possible application for the sequential MDP model proposed in this paper is the market investment. Suppose that an investor has certain amount of resources to invest in an open market (a market where prices change in a continuous manner like currency exchange). The investor knows on average the price values (for example low season and high season prices). However, in a given period known to have high prices, the investor observed that the market is announcing lower prices than usual. Should he invest in that period or should he wait to the next low season prices?
Again typical MDP solution would give before hand policies that do not take into account observed outcomes. An MDP solution in this scenario would behave inefficiently. 

\subsection{Model}
The new sequentially observed MDP model has the following components:  
\begin{itemize}
\item The current state and the transition probabilities are known, i.e., the probability of transitioning  from any state $i$ to another state $j$ when an action $a$ is taken. 
\item At a given decision epoch $t$, the agent observes the possible next state if action $a_1$ was taken, but only knows the transition probabilities for the rest of the actions. The agent must either accept or reject the transition due to $a_1$. Accepting the transition means the agent chose action $a_1$ at time epoch $t$, rejecting the transition means that the agent will not choose action $a_1$ and the action must be chosen from the remaining possible actions.
\item Only after the rejection of $a_1$, the agent can observe the deterministic transition if $a_2$ is taken, and only knows the probability of  transitions for the remaining actions. Again, accepting the transition means the agent has chosen action $a_2$ at decision epoch $t$. Rejecting the transition means that the agent will not choose action $a_1$ or $a_2$ and the action must be chosen from the remaining possible actions.
\item The procedure is repeated till the action $m-1$. If the observed transition due to $a_{m-1}$ was rejected, then the agent has no choice and must choose $a_m$ (without observing its corresponding transition). We say that the system is at \emph{phase} $k$ if the agent observes the transition due to action $a_k$ and has not yet made a decision (to reject or accept it).
\item Once any action is taken (accepting an observed transition or rejecting all observed transitions), the next decision epoch starts.
\end{itemize}    

Note that a typical MDP decision-making algorithm can be adopted as follows:   the decision policy is computed by using a standard MDP solution method \cite{Puterman:1994Mar} by ignoring the observed transitions. For example, if the optimal policy was to select $a_i^*$ at decision epoch $t$, then the agent would discard the observed transitions  for $a_1, \dots, a_{i-1}$, and would accept any observed transition for $a_{i}$ action. Our goal in this paper is  to take advantage of the  additional observed transitions to increase the expected rewards.

\begin{proof}[\bf Remark]
The proposed model is different from the well known ``secretary problem'' in MDP literature \cite{Ferguson:1989Who,Babaioff:2007Akn}. In the secretary problem, a fixed number of people are interviewed for a job in a sequential manner, and based on the (observed) rank of the current  interviewed candidates, a decision should be taken whether to accept or reject the last interviewed candidate. The main difference with the sequentially observed MDPs is that in the ``secretary problem'', observing a candidate changes the probability of future transitions (because of the correlation between the events).  However, in our model an observation is independent from the further environmental dynamics (i.e., observing a transition at a given phase does not change the transition probabilities for next phases or epochs).  Another fundamental difference is that our model does not necessarily have a stopping time, and the horizon can go to infinity which is not possible for the secretary problem.
\end{proof}

\section{Defining the MDP}
\subsection{States and Actions}
 Let the set $S=\{ 1, \dots , n\}$ be the set of states having a cardinality $|S|=n$. Let us define $\mathcal{A}_s=\{ 1, \dots ,m\}$  to be the set of actions available in state $s$ (without loss of generality the number of actions does not change with the state, i.e., $|\mathcal{A}_s|=m$ for any $s\in S$). We consider a discrete-time system where actions are taken at different decision epochs. Let $s(t)$ and $a(t)$ be respectively the state and action at the $t$-th decision epoch. 

\subsection{Decision Rule and Policy}
We define a decision rule $D_t$ at time $t$ to be the following randomized function  $$D_t : S \rightarrow \mathcal{A}_S$$ that defines for every state $s\in S$ a random variable $D_t(s)\in \mathcal{A}_{s}$ with some probability distribution defined over $\mathcal{P}(\mathcal{A}_s)$. In typical MDPs, the decision variables are directly the probability distribution of this random variable $p_{i}(a,t)=\text{Prob}[D_t=a|s(t)=i]$ for any action $a\in \mathcal{A}_i$ and given any state $i$.  In the sequential MDP, the decision variables are whether to accept or reject a given transition at phase $k$. We then define the decision variables as follows: 
\begin{align*}
P_i(j,k,t)&=\text{Prob\big [Accepting observed transition to state $j$ $|$}\\
&\hspace*{0.5cm}\text{System is in state $i$ and phase $k$ and epoch $t$\big ]}.
\end{align*}
Since there are only $m-1$ phases, we assume $P_i(j,k,t)=1$ if $k=m$.
In this new formulation, the order of the actions is important.

Let $$\pi =(D_1, D_2, \dots, D_{N-1})$$ be the policy for the decision making process given that there are $N-1$ decision epochs. Then in typical MDP, the decision $D_t$ is defined by the independent vector variables $\{\mathbf{p}_1(t), \dots,\mathbf{p}_n(t)\}$ where $\mathbf{p}_i$ is the vector having the probabilities $p_{i}(a,t)\geq 0$ for all $a\in \mathcal{A}_i$ and decision epoch $t$ and such that $\sum _ap_{i}(a,t)=1$. In the sequential MDP, the  decision $D_t$ is defined by the independent matrix variables $\{P_1(t), \dots, P_n(t)\}$ where $P_i(t)$ is the matrix having the probabilities $P_i(j,k,t)\in [0,1]$ for all destination states $j\in S$,  for $k=1,\dots, m$, and decision epoch $t$. For notation simplicity we will drop the index $t$ from the notation when there is no confusion and variables are denoted simply by $P_i$; the upcoming results  are for  time dependent cases.
 Note that this decision rule has a Markovian property because it depends only on the current state. Indeed this paper considers only Markovian policies, history dependent policies \cite{Puterman:1994Mar} are not considered.  

\subsection{Rewards}
Given a state $s\in S$ and action $a\in A$, we define the reward $r_t(s,a)\in \mathbb{R}$ to be any  real number and let $\mathcal{R}$ to be the set having these values. 
With a little abuse of notation, we define the expected  reward for a given  decision rule $D_t$ at time $t$ to be 
\begin{equation}\label{eq:reward}
r_t(s)=\mathbb{E}[r_t(s,D_t(s))]=\sum _{a\in \mathcal{A}_{s}}p_s(a)r_t(s,a),
\end{equation} and the vector $\mathbf{r}_t\in \mathbb{R}^n$ to be the vector with the expected rewards for each state. Given there are $N-1$ decision epochs, then there are $N$ reward stages and the final stage reward is given by $r_N(s)$ (or $\mathbf{r}_N$ the vector having as its elements the final reward at a given state).

\subsection{State Transitions}
We now define the transition probabilities as follows, $G_i(j, k,t)= \text{Prob}[s(t+1)=j|s(t)=i, \text{phase }  k]$, and $G_i(t)$ be the corresponding matrix (for simplicity we will drop the index $t$ from the notation when there is no confusion and transitions are denoted simply by $G_i$). Let $\mathcal{G}$ be the set having these transition matrices. Let's define an intermediate variable $q_i(a_k)$ for notational convenience, which is the probability of choosing action $a_k$ given that the previous actions $a_1,...,a_{k-1}$ are rejected 
$$q_i(a_k)=\sum _{j\in S}G_i(j, k)P_i(j,k).$$
Then  the probability that the agent chooses action $a_k$ is the probability that the agent rejects the first $k\!-\!1$ actions (i.e., $\prod _{l=1}^{k-1}\left(1-q_i(a_l)\right)$) and then accepts the $k$-th action (i.e., $q_i(a_k)$):
\begin{equation}\label{eq:pia}
p_i(a_k)=\left(\prod _{l=1}^{k-1}\left(1-q_i(a_l)\right)\right)q_i(a_k)\text{ if }1\leq k\leq m,
\end{equation}
where, by convention, $\prod _{l=1}^{k-1}\left(1-q_i(a_l)\right)=1$ if $k=1$. We observe that $q_i(a_k)\!=\!1$ if $k\!=\!m$. The above relation shows that the decision variables due to the typical MDP ($p_i(a_k)$ for $k=1,\dots, m$ and $i=1,\dots, n$) are a non-convex function of the decision variables of the sequential MDP ($P_i$ for $i=1,\dots, n$).
The transition probability from a state $i$ to a state $j$ is given by the probability to reach phase $k$ and transition to state $j$ is accepted, i.e.,
\begin{align}
M_t(j,i)&=\text{Prob}[s_{t+1}=j|s_t=i]\nonumber\\
&=\sum _{k=1}^m\left(\prod _{l=1}^{k-1}\left(1-q_i(a_l)\right)\right)G_i(j,k)P_i(j,k).\label{eq:Mji}
\end{align}
Also in this case, the transition is not linear in the decision variables for the sequentially observed MDP.
Let  $x_i(t)=\text{Prob}[s_t=i|s_1]$ be the probability of being at state $i$ at time $t$, and $\mathbf{x}(t)\in \mathbb{R}^m$ to be the vector of these probabilities. Then the system evolves according to the following recursive equation:
$$\mathbf{x}(t+1)=M_t \mathbf{x}(t),$$
where $M_t$ (or simply $M$) is the matrix having the elements $M_t(j,i)$ (or simply $M(j,i)$). It is important to note that the $i$-th column of $M$ (its transpose is denoted by $M^{Ti}$) is a function of the decision variables in the matrix $P_i$ only (i.e., independent of the variables of the matrices $P_{s}$ for $s\neq i$). 

\subsection{Markov Decision Processes (MDPs)}
Let $\gamma \in [0,1]$ be the discount factor, which represents the importance of a current reward in comparison to future possible rewards. We will consider $\gamma =1$ throughout  the paper, but the results  are not affected and remain applicable after a suitable scaling when $\gamma <1$. 

A discrete MDP is a 5-tuple $(S, A_S, \mathcal{G}, \mathcal{R}, \gamma )$ where $S$ is a finite set of states, $A_s$ is a finite set of actions available for  state $s$, $\mathcal{G}$ is the set that contains the transition probabilities   given the current state and current action, and $\mathcal{R}$ is the set of rewards at a given time epoch due to the current state and  action.

\subsection{Performance Metric}
 For a policy to be better than another policy we need to define a performance metric. We will use the expected discounted total reward for our performance study, 
 $$v_N^\pi=\mathbb{E}_{\mathbf{x}(1)}\left[\sum_{t=1}^{N-1}r_t(X_t,D_{t}(X_t))+ r_N(X_N)\right],$$
where $X_t$ is the state at decision epoch $t$ and the expectation is conditioned on a probability distribution over the initial states (i.e., $\mathbf{x}(1)\in \mathcal{P}(S)$ where $x_i(1)=\text{Prob}[s_1=i]$). It is worth noting that  both $X_t$ and $D_{t}(X_t)$ are random variables in the above expression. 

\subsection{Optimal Markovian Policy}
The optimal policy $\pi ^*$ is given as the policy that maximizes the performance measure, $\pi ^* = \text{argmax}_{\pi} v_N^\pi$, and $v^*_N$ to be the optimal value, i.e., $v_N^* = \text{max}_{\pi} v_N^\pi$.
Note that  the optimization variables of the above maximization are $P_1(t), \dots, P_n(t)$ for $t=1, \dots, N-1$.\footnote{Since $v^\pi_N$ is continuous in the decision variables that belong to a closed and bounded set, then the $\max$ is always attained and $\text{argmax}$ is well defined.} For the typical MDP, the \emph{backward induction} algorithm \cite[p.~92]{Puterman:1994Mar} gives the optimal policy as well as the optimal value. However, in our new model the optimization variables are different and another algorithm for finding optimal policies is needed. In the following sections, we will give such an algorithm for the sequential MDP (SMDP) and we will show its optimality using Bellman equations of dynamic programming.

\section{Dynamic Programming (DP)  Approach for  MDPs}
In this section, we  transform the MDP problem into a deterministic Dynamic Programming (DP) problem and use this approach to devise an efficient algorithm for finding optimal policies of the new introduced model.
First note  that the performance metric can be written as follows:
\begin{align*}
v_N^\pi&=\mathbb{E}_{\mathbf{x}(1)}[(\sum_{t=1}^{N-1}r_t(X_t,D_{t}(X_t)))+ r_t(X_N)]\\
&=\sum _{t=1}^{N-1}\mathbb{E}_{\mathbf{x}(1)}[r_t(X_t,D_{t}(X_t))] + \mathbb{E}_{\mathbf{x}(1)}[r_t(X_N)]\\
&=\sum _{t=1}^{N-1}\mathbb{E}_{\mathbf{x}(1)}[\mathbb{E}_{X_t}[r_t(X_t,D_{t}(X_t))]] + \mathbb{E}_{\mathbf{x}(1)}[\mathbb{E}_{X_N}[r_t(X_N)]]\\
&=\sum _{t=1}^{N}\mathbb{E}_{\mathbf{x}(1)}[\mathbf{e}_{X_t}^T\mathbf{r}_t]=\sum _{t=1}^{N}\mathbf{x}(t)^T\mathbf{r}_t,
\end{align*}
where $\mathbf{e}_s$ is the vector of all zeros except a value $1$ at the  position $s$. The last equality utilized the fact that $\mathbb{E}_{\mathbf{x}(1)}[X_t]=\mathbf{x}(t)$.

We can now give the DP  formulation. For notation simplicity, let $\mathbf{x}_t=\mathbf{x}(t)$. The discrete-time dynamical system describing the evolution of the density $\mathbf{x}_t$ can then be given by
\begin{align*}
\mathbf{x}_{t+1} &=f_t(\mathbf{x}_t,P_1(t), \dots, P_n(t)) \text{ for } t=1,\dots , N-1,
\end{align*} 
such that $f_t(\mathbf{x}_t,P_1(t), \dots, P_n(t)) = M_t\mathbf{x}_t$ where $M_t=M_t(P_1(t), \dots, P_n(t))$ is the transition matrix a function of the optimization variables. The elements of the $i$-th \emph{column} in $M_t$ are  functions of  only the elements in $P_i(t)$ matrix as mentioned earlier. The above dynamics show that the probability distribution  evolves deterministically. Our policy $\pi=(D_1, \dots, D_{N-1})$ consists of a sequence of functions that map states $\mathbf{x}_t$ into controls $P_i(t)=D_{i,t}(\mathbf{x}_t)$ for all $i$ in such a way that $D_{i,t}(\mathbf{x}_t) \in \mathcal{C}(\mathbf{x}_t)$ where $\mathcal{C}(\mathbf{x}_t)$ is the set of constraints on the control. Since the only constraints on the decision variables are that they are restricted to the interval $[0,1]$, then $\mathcal{C}(\mathbf{x}_t)$ is independent of $\mathbf{x}_t$ and all admissible controls belong to the same convex set $\mathcal{C}$ for any given state.

The additive reward per stage is defined as $g_N(\mathbf{x}_N)=\mathbf{x}_N^T\mathbf{r}_N$ and 
$$g_t(\mathbf{x}_t, P_1(t), \dots, P_n(t))=\mathbf{x}_t^T\mathbf{r}_t , \text{ for } t=1, \dots, N-1.$$ 
The dynamic programming then calculates the optimal value $v^*_N$ (and policy $\pi ^*)$ by running Algorithm \ref{alg2}\cite[Proposition 1.3.1, p.~23]{Bertsekas:2005Dyn}.
\begin{algorithm}
\caption{Dynamic Programming}
\label{alg2}
\begin{algorithmic} [1]
 \STATE Start with $J_N(\mathbf{x})=g_N(\mathbf{x})$
\STATE for $t=N-1,\dots , 1$
\begin{align*}
J_t(\mathbf{x})&=\max _{P_1(t),\dots, P_n(t)\in \mathcal{C}(\mathbf{x})}\Big \{g_t(\mathbf{x}, P_1(t),\dots, P_n(t))+\\
&\hspace{2cm}J_{t+1}(f_t(\mathbf{x}, P_1(t),\dots, P_n(t)))\Big \}.
\end{align*} 
\STATE {\bf Result:} $J_1(\mathbf{x})=v^*_N$.
\end{algorithmic}
\end{algorithm}

\begin{proof}[\bf Remark]
There are several  difficulties in applying the DP Algorithm \ref{alg2}. Note that in the term $J_{t+1}(f_t(\mathbf{x},P_1,\dots, P_n))$ used in the algorithm $P_i$s are the optimization variables. 
For a given $P_i$ and $\mathbf{x}$, numerical methods can be used to compute the value of  $J_{t+1}$. 
But since $P_i$ itself is an optimization variable,  the solution of the optimization problem in line $2$ of Algorithm \ref{alg2} can be very hard. 
In some special cases, for example when $J_t(\mathbf{x})$  can be expressed analytically  in a closed from, the solution complexity can be reduced significantly, as we will show next for the sequential MDP problems. 
\end{proof}


\subsection{Backward Induction for the sequential MDP model}
 This section presents  the optimal backward induction algorithm for solving the sequential MDP by using the dynamic programming approach. The set of admissible controls at time $t$ is given by $\mathcal{C}(\mathbf{x}_t)=\mathcal{C}$ defined as follows:
 $$0\leq P_i(j,k,t)\leq 1 \text{ for all }i\in S,j\in S, k\in \mathcal{A}_i $$
  Using the dynamic programming Algorithm \ref{alg2}, we can now give the following proposition:
\begin{prop}\label{prop:ClosedFormJ}
The term $J_t(\mathbf{x})$ in the dynamic programming algorithm for the sequential MDP  has the following closed-form solution:
$$J_t(\mathbf{x})=\mathbf{x}^TV^*_t,$$
where $V^*_t$ is a vector that satisfies the following recursion,
$V^*_N=\mathbf{r}_N$ and for $t=N-1, \dots, 1$ we have
$$V^*_t(i)=\max _{P_i(t)}\left\{r_t(i)+M^{Ti}_tV^*_{t+1}\right\} \text{ for } i=1,\dots, n.$$
\end{prop}
\begin{proof}
We will show that by induction. From the definition of $g_N(.)$ we have the base case satisfied (i.e., $J_t(\mathbf{x})=\mathbf{x}^T\mathbf{r}_N=\mathbf{x}^TV^*_{N}$. Suppose the hypothesis is true from $N-1,\dots, t+1$, then we show it is true for $t$. From the DP algorithm, we can write  
\begin{align}
J_{t}(\mathbf{x})&= \max _{P_1(t),\dots, P_n(t)\in \mathcal{C}}\big \{\mathbf{x}^T\mathbf{r}_t +J_{t+1}(M_t\mathbf{x})\big \}\label{eq:line1} \\
&=\max _{P_1(t),\dots, P_n(t)\in \mathcal{C}} \{\mathbf{x}^T\mathbf{r}_t+\mathbf{x}^TM_t^TV^*_{t+1}\}\label{eq:line2} \\
&=\max _{P_1(t),\dots, P_n(t)\in \mathcal{C}} \{ \sum _i x_i (r_t(i)+M^{Ti}_tV^*_{t+1})\}\label{eq:line3}\\
&=\sum _i x_i\left(\max _{P_i(t)\in C}\left\{ r_t(i)+M^{Ti}_tV^*_{t+1}\right\}\right)\label{eq:line4}
\end{align}
where $M^{Ti}_t$ indicates the transpose of the $i$-th column of $M$ which is a function of the decision variables of the $P_i$ matrix only. The transition from \eqref{eq:line1} to \eqref{eq:line2} is due to the induction assumption, and the transition from \eqref{eq:line3} to \eqref{eq:line4}  is because $x_i\geq 0$ for all $i$ and the function is  separable in terms of the optimization variables. The maximization inside the parenthesis is nothing but $V^*_{t}(i)$, then $J_{t}(\mathbf{x})=\sum _i x_iV^*_{t}(i) = \mathbf{x}^T V^*_{t}$ and this ends the proof.
\end{proof}

 Notice that $J_{t}(\mathbf{x})$ has a closed-form equation as function of $\mathbf{x}$ and so it suffices  the calculation of $V^*_{t}$ for $t=N,\dots ,1$  for finding the optimal value of the MDP given by $v^*_N=J_1(\mathbf{x}_1)=\mathbf{x}_1^TV^*_{1}$. The backward induction algorithm is given in Algorithm \ref{algSMDP}.
 
 \begin{algorithm}
\caption{Backward Induction: Sequential MDP Optimal Policy}
\label{algSMDP}
\begin{algorithmic} [1]
 \STATE {\bf Definitions:} For any state $s\in S$, we define $V_t^\pi(s)=\mathbb{E}_{\mathbf{x}_t=\mathbf{e}_s}\left[\sum_{k=t}^{N-1}r_k(X_k,D_k(X_k))+ r_k(X_N) \right]$ and $V_t^*(s)=\text{max}_{\pi} V_t^\pi$ given that $s_t=s$.
\STATE Start with $V^*_N(s)=r_N(s)$
\STATE for $t=N-1, \dots, 1$ given $V_{t+1}^*$ and for $s=1,\dots, n$ calculate the optimal value
  $$V_{t}^*(s)=\max _{P_s\in \mathcal{C}} \left\{ r_t(s) + \sum _{j\in S}M_t(j,s)V_{t+1}^*(j)\right\}$$
  and the optimal policy $P^*_s(t)$ given by: 
  $$P^*_s(t)=\underset{P_s\in \mathcal{C}}{\text{argmax}} \left\{ r_t(s) + \sum _{j\in S}M_t(j,s)V_{t+1}^*(j)\right\}$$
\STATE {\bf Result:} $V_1^*(s_1)=v_N^*$ where $s_1$ is the initial state.
\end{algorithmic}
\end{algorithm}

 {\bf Remark:} We want to stress two points about the algorithm. First, the policy calculated by Algorithm \ref{algSMDP} is optimal (maximizing the total expected reward) because of line 3 in Algorithm \ref{alg2} and Proposition \ref{prop:ClosedFormJ}. Second,   $r_t(s)$ and $M_t(j,s)$ are both  functions of the decision variables in $P_i$. In typical MDPs, these values are simply linear in the decision variables. However, in the proposed  sequential MDP  model, these values are non-convex in the decision variables and a further processing is needed for efficient implementation of the algorithm, which is discussed next. 
  
\subsection{Efficient Implementation of Algorithm \ref{algSMDP}}
In the internal loop of Algorithm \ref{algSMDP}, the optimal value at a given decision epoch $t$ is given by the following equation: 
\begin{equation}\label{eq:Vtstar}
V_{t}^*(i)=\max _{P_i\in \mathcal{C}}V_{t}(i),
\end{equation}
where $V_{t}(i)= r_t(i) + \sum _{j\in S}M_t(j,i)V_{t+1}^*(j)$.
In this formulation, $r_t(i)$ and $M_t(j,i)$ are  functions of the decision variable $P_i(t)$, for given state   $i$  and time epoch  $t$. In particular, the explicit expression can be deduced from Eq.~\eqref{eq:reward}, Eq.~\eqref{eq:pia}, and Eq.~\eqref{eq:Mji} as follows:
\begin{align*}
r_t(i)&=\sum _{a\in \mathcal{A}_{s}}p_i(a)r_t(i,a)\\
&=\sum _{k=1}^m\left(\prod _{l=1}^{k-1}\left(1-q_i(a_l)\right)\right)q_i(a_k)r_t(i,a_k).
\end{align*}
and 
\begin{equation}
M_t(j,i)=\sum _{k=1}^m\left(\prod _{l=1}^{k-1}\left(1-q_i(a_l)\right)\right)G_i(j,k)P_i(j,k),
\end{equation}
where $q_i(a_k)=\sum _{j}G_i(j,k)P_i(j,k)$. By substituting these equations in the expression of  $V_{t}(i)$, we obtain
\begin{align}
V_{t}(i)&=r_t(i) + \sum _{j\in S}M_t(j,i)V_{t+1}^*(j)\\
&\hspace*{-0.5cm}=\sum _{k=1}^m\left(\prod _{l=1}^{k-1}\left(1-q_i(a_l)\right)\right)\left(\sum _{j=1}^nG_i(j,k)P_i(j,k)\right)r_t(i,a_k)\nonumber\\
&\hspace*{-0.4cm}+\sum _{j=1}^n\left(\sum _{k=1}^m\left(\prod _{l=1}^{k-1}\left(1-q_i(a_l)\right)\right)G_i(j,k)P_i(j,k)V_{t+1}^*(j)\right)\\
&\hspace*{-0.5cm}=\sum_{k=1}^m\sum_{j=1}^n\left(r_t(i,a_k)+V_{t+1}^*(j)\right)G_i(j,k)X_i(j,k)\\
&\hspace*{-0.5cm}=\sum_{k=1}^m\sum_{j=1}^nH_i(j,k)X_i(j,k).
\end{align}
where 
\begin{eqnarray}
X_i(j,k)&:=&\prod _{l=1}^{k-1}\left(1-q_i(a_l)\right)P_i(j,k) \label{eq:X} \\ 
H_i(j,k)&:=&\left(r_t(i,a_k)+V_{t+1}^*(j)\right)G_i(j,k). \nonumber
\end{eqnarray}
Note that $H_i(j,k)$ is independent of the decision variables. 

For efficient implementation of the algorithm, it remains to show what conditions should $X_i(j,k)$ satisfy so that the mapping $X_i(j,k)=\prod _{l=1}^{k-1}\left(1-q_i(a_l)\right)P_i(j,k)$ is invertable. Notice that if $q_i(a_l)\neq 1$ for $l=1,\dots, m-1$, then the mapping is one-to-one mapping and we will give the expression for $P_i$ in terms of $X_i$ shortly after. If there exists $l$ such that $q_i(a_l)=1$, then the phases $k>l_{min}$ are not reached because an earlier action must necessarily be accepted where $l_{min}=\min \{l|q_i(a_l)=1\}$. This means that $V_t(i)$ is independent of $P_i(j,k)$ when $k>l_{min}$ (i.e., the optimal value is not affected by these variables) and without loss of generality we can consider $P_i(j,k)=1$ for $j=1,\dots, n$ and $k=l_{min}+1,\dots ,m$. 

 We can give now the expression of $P_i$ in terms of $X_i$ by the following lemma:
\begin{lemma} 
For a given state $i$, the following equation holds for $X_i(j,k)$,  $j=1,\dots, n$ and $k=1,\dots, m$, in Eq. \eqref{eq:X}:
\begin{equation}
X_i(j,k)=\left(1-\sum _{l=1}^{k-1}\sum_{s=1}^nG_i(s,l)X_i(s,l)\right)P_i(j,k).
\end{equation}
\end{lemma}
\begin{proof}
We will prove this lemma by showing that $\prod _{l=1}^{k-1}\left(1-q_i(a_l)\right)=1-\sum _{l=1}^{k-1}\sum_{s=1}^nG_i(s,l)X_i(s,l)$ by induction. It is true for $k=2$ by the definition of $q_i(a_l)$. Suppose it is true till $k-2$, and let us show it true for $k-1$. We have 
\begin{align*}
\prod _{l=1}^{k-1}\left(1-q_i(a_l)\right)&=\left(\prod _{l=1}^{k-2}\left(1-q_i(a_l)\right)\right)\left(1-q_i(a_{k-1})\right)\\
&=\Big(\prod _{l=1}^{k-2}\left(1-q_i(a_l)\right)\\
&\hspace*{0.5cm}-\sum _{s}G_i(s,k-1)X_i(s,k-1)\Big)\\
&=1-\sum _{l=1}^{k-1}\sum_{s=1}^nG_i(s,l)X_i(s,l).
\end{align*}
where the last equality uses the induction hypothesis. 
\end{proof}
It remains to derive the constraints on $X_i(j,k)$ when $P_i\in \mathcal{C}$. Since $P_i(j,k)\in [0,1]$ for all  $j=1,\dots ,n$ and $k=1,\dots, m-1$, then we can derive the following conditions:
$$0\leq X_i(j,k)\leq 1-\sum _{l=1}^{k-1}\sum_{s=1}^nG_i(s,l)X_i(s,l),$$
and since by definition $P_i(j,m)=1$ for all $j=1,\dots, n$:
$$X_i(j,m)=1-\sum _{l=1}^{m-1}\sum_{s=1}^nG_i(s,l)X_i(s,l).$$

As a result, $V^*_t(i)$ is the solution of the following linear program
\begin{equation}\label{eq:LPd}
\begin{aligned}
\underset{X_i}{\text{maximize}}& \hspace*{0.5cm} \sum_{k=1}^m\sum_{j=1}^nH_i(j,k)X_i(j,k)\\
\text{subject to:}& \text{ for }j=1,\dots, n \text{ and }k=1,\dots, m-1\\
&\hspace*{-1cm} 0\leq X_i(j,k)\leq 1-\sum _{l=1}^{k-1}\sum_{s=1}^nG_i(s,l)X_i(s,l),\\
&\hspace*{-1cm} X_i(j,m)=1-\sum _{l=1}^{m-1}\sum_{s=1}^nG_i(s,l)X_i(s,l).
\end{aligned}
\end{equation}

To write it in matrix form, let $\mathbf{1}_n$ be the vector of all ones and dimension $n$, $J=\mathbf{1}_n\mathbf{1}_n^T$, and $B$ be a constant $m\times m$ matrix defined as $B(l,k)=1$ if $k>l$ and $B(l,k)=0$ otherwise. 
\begin{lemma}
The linear program \eqref{eq:LPd} can be written in matrix-form as follows:
\begin{equation}\label{eq:LP}
\begin{aligned}
\underset{X_i}{\text{maximize}}& \hspace*{0.5cm} Tr(H_i^TX_i)\\
\text{subject to}&\hspace*{0.5cm} 0\leq X_i+J(G_i\odot X_i)B\leq \mathbf{1}_n\mathbf{1}_m^T\\
& \left(X_i+J(G_i\odot X_i)B\right)\mathbf{e}_m= \mathbf{1}_n.
\end{aligned}
\end{equation}
\end{lemma}

Let $z(k)=1-\sum _{l=1}^{k-1}\sum_{s=1}^nG_i(s,l)X_i^*(s,l)$ if $k=2,\dots, m$ and $z(1)=1$. The following proposition summarizes our results
\begin{prop}
For a given  decision epoch $t$ and  state $i$,  the optimal value and optimal policy terms in Algorithm~\ref{algSMDP} are given  by
$$V^*_t(i)=Tr(H_i^TX_i^*),$$
and for $j=1,\dots, n$ and $k=1, \dots, m$
\begin{equation}
P^*_i(j,k,t)=
\begin{cases}
X^*_i(j,k)/z(k) &\text{ if }z(k)>0,\\
1 &\text{ else }.
\end{cases}
\end{equation}
where $X_i^*$ is the solution of  the linear program \eqref{eq:LP}.
\end{prop}
\begin{proof}
The proof is based on the fact that the linear program in the decision variables $X_i$ is equivalent to the original optimization over the $P_i$ variables because the mapping between the variables is one-to-one mapping when considering the additional (redundant) constraints: $P_i(j,k)=1$ for $j=1,\dots, n$ and $k=l_{min}+1,\dots ,m$.
\end{proof}
\section{Simulations}
This section presents a simulation example to demonstrate  the proposed policy synthesis  method for  the  MDPs with sequentially observed transitions. In this application, autonomous vehicles (agents)  explore a region $F$, which can be partitioned into $n$ disjoint subregions (or bins) $F_i$ for  $i=1, \dots, n$ such that $F=\cup _iF_i$ \cite{Acikmese:2012Ama, behcet_tac15}.  We can model the system as an MDP where the states of agents are their  bin locations and the actions of a vehicle are defined by the possible transitions to neighboring bins.  Each vehicle collects rewards while traversing the area where, due to the stochastic environment, transitions are stochastic (i.e., even if the vehicle's command is to move to ``right", the environment can send the vehicle to ``left"). 
 In particular, with probability $0.6$ the given command will lead to the desired bin, while with  probability $0.4$   the agent would land on another neighboring bin.
We assume a region describe by a 10 by 10 grid. Each vehicle has 5 possible actions: ``up'',  ``down'',  ``left'',  ``right'', and  ``stay''. When the vehicle is on the boundary, we set the probability of actions that cause transition outside of   the domain to zero.  The total number of states is  100 with 5 actions, and a decision time horizon $N\!=\!10$. The reward vectors $R_t$ for $t=1,\dots ,N-1$ and $R_N$  are chosen randomly with entries  in the interval $[0,100]$. Since any feasible policy for  a standard  MDP  is also a feasible solution for the proposed sequential model (i.e., $\pi_{MDP} \subseteq \pi_{SMDP}$), then the following holds:
$$v_N^{\pi ^*_{MDP}}\leq v_N^{\pi ^*_{SMDP}}.$$
Figure \ref{Utility} shows the difference in values due to optimal policies of the standard  MDP model and the proposed sequential MDP (i.e., $v_N^{\pi ^*_{SMDP}}-v_N^{\pi ^*_{MDP}}$). The figure shows that, depending on initial state, the new model can have significant improvement by utilizing the additional information  (observing the transitions before deciding on actions). 

\begin{figure}
\begin{center}
\includegraphics[scale=0.5]{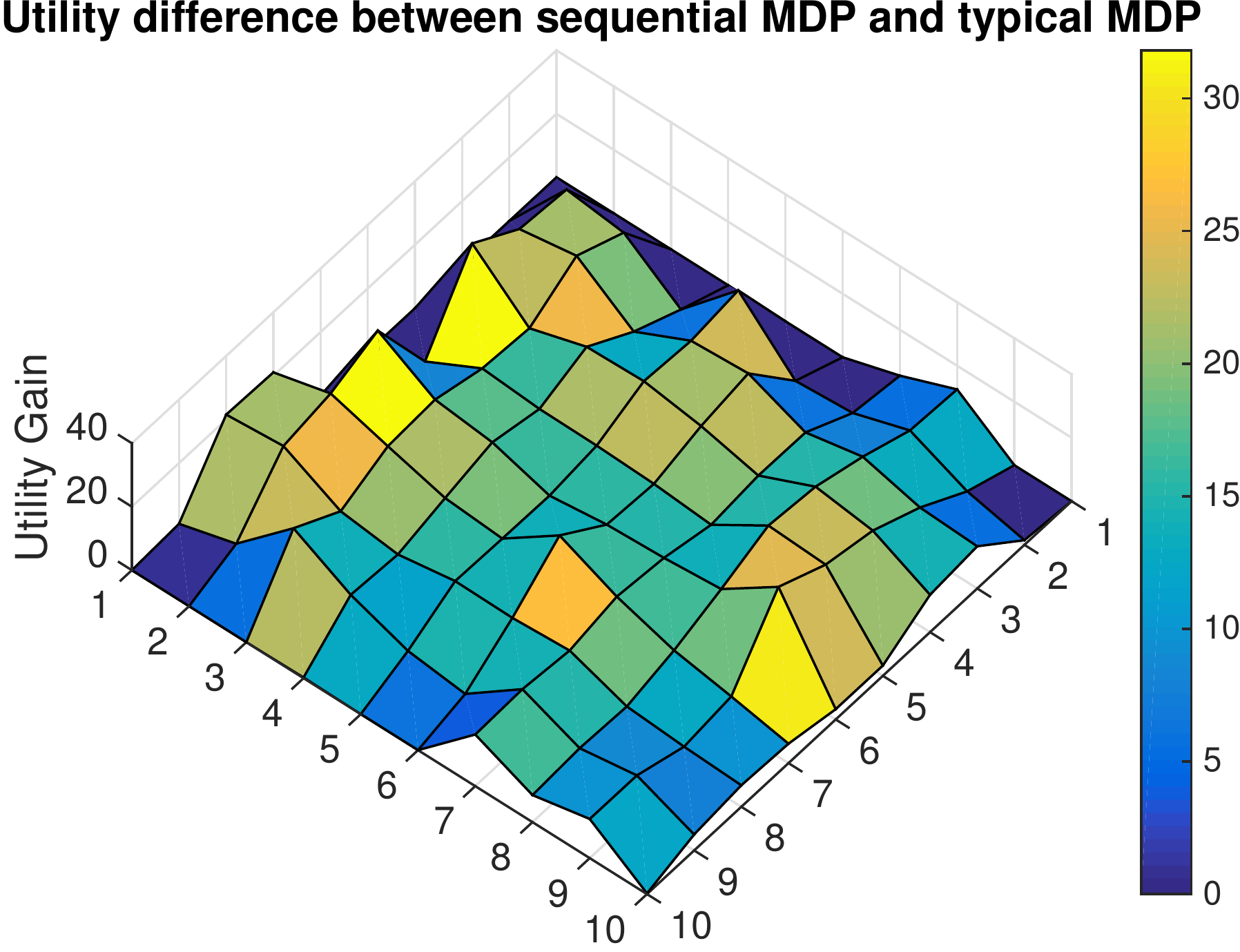}
\caption{The figure shows the difference in the utility (optimal value) of the sequential MDP strategy that takes advantage  of the observed transitions and the standard  MDP that does not use this extra information. The figure shows that the difference in the utility depends on the initial position of agents. Some bins can give a higher than expected reward than other bins.
} 
\label{Utility}
\vspace{-20pt}
\end{center}
\end{figure}

\section{Conclusion}
This paper   introduces  a novel model for MDPs that incorporates additional observations  on the transitions for a given action in a sequential manner. This model achieves better expected total rewards than the optimal policies for the standard  MDP models studied in the literature due to the utilization of   additional information. We also propose an efficient algorithm based on linear programming that allows offline  calculations of these optimal policies.

\bibliographystyle{./IEEEtran}
\bibliography{./SMDP_bib,./IEEEabrv}
\end{document}